\documentclass[11pt]{article}
\usepackage{amsfonts}
\usepackage{graphics,amssymb,amsthm,amsmath,epsf}
\usepackage{graphicx}
\usepackage{subfigure}

\usepackage[round]{natbib}
\usepackage{xcolor}

\newtheorem{lemma}{Lemma}[section]

\numberwithin{equation}{section}
\newtheorem{remark}{Remark}[section]

\allowdisplaybreaks

\begin{document}

\title { Preliminary Test and ‎‎Shrinkage Estimations of Scale Parameters for Two Exponential Distributions based on Record Values}

\author{H. Zakerzadeh\thanks{Corresponding: hzaker@yazd.ac.ir} , A. A. Jafari, M. Karimi  \\
{\small Department of Statistics, Yazd University, Yazd, Iran}}

\date{}
\maketitle

\begin{abstract}
The exponential distribution is applied in a very wide variety of statistical procedures. Among the most prominent applications are those in the field of life testing and reliability theory. When there are two record samples available for estimating the scale parameter, a preliminary test is usually used to determine whether to pool the samples or use the individual sample.
In this paper, the preliminary test estimator and shrinkage estimator are studied. The optimum level of significance for preliminary test estimation and the optimum values of shrinkage coefficient are obtained based on minimax regret criterion under the weighted square error loss function.
\end{abstract}


‎\section{Introduction}‎
Record values are of interest and important in many real life applications involving data relating to meteorology, sport, economics and life testing. Sometimes, the experimenter has some knowledge about the parameter of interest, either from past experience, or from experience with similar situations. This prior information may be incorporated in the estimation process using a preliminary test estimator
\citep{oh-to-78,to-wa-75,sa-hi-73,chiou-90}
and improved the estimation process. Preliminary test and shrinkage estimation in the one and two-parameter
exponential distributions are considered by many authors including
\cite{pandey-83-shr}
and
\cite{ch-ha-89}.

Let $X_1,X_2,.....$ be a sequence of independent and identically distributed random variables having the same distribution. An observation $X_j$ will be called an upper record value if exceeds in value all of the preceding observations, i.e., if $X_j>X_i$, for every $i<j$. The sequence of record times is defined as follows: $T_1=1$ with probability 1 and, $T_n=\min\{ j : X_j > X_{T_{n-1}}‎\}$ for $n\geq2‎$. A sequence of upper record values is defined by $X_{U(n)}=X_{T_n}$‎,  $n=1,2,\dots$‎.
For details on record values and other interesting topics related to records see
\cite{ahsanullah-95} and \cite{ar-ba-na-98}.

In this paper the estimation of the scale parameter in two exponential distributions based on record values is studied. When two record samples from exponential distributions are available, the question of whether to pool or not to pool these two record samples is often determined via a preliminary test. If the test is not statistically significant, the pooled estimator is used; otherwise, the likelihood estimator is used. The optimum level of statistical significance for the usual preliminary test estimator are obtained in Section \ref{sec.prt} by using the minimax regret criterion.
A shrinkage estimator is also considered. 
The optimum values of shrinkage coefficients for the shrinkage estimator are obtained in Section \ref{sec.shr}.‎
The proposed estimators are illustrated  using an example, .and are compared using simulation in Section \ref{sec.num}.

\section{Preliminary test estimation}
\label{sec.prt}

Let $X_1,X_2,\dots$ and $Y_1,Y_2,\dots$ be  sequences of independent and identically distributed random variables from two exponential models with the following probability density function:
\begin{eqnarray*}
 f_X(x)=\frac{1}{\theta_1} \exp ( - \frac{x}{\theta_1} ) \ \ \ \ x > 0,    \ \ \ \
 f_Y(y)=\frac{1}{\theta_2} \exp ( - \frac{y}{\theta_2} ) \ \ \ \ y > 0.
 \end{eqnarray*}
Also,  suppose that we observe $n_1$ and $n_2$ upper record values from these two sequences as $\boldsymbol X=(X_{U(1)}, \dots, X_{U(n_1)})$ and
 $\boldsymbol Y=(Y_{U(1)}, \dots, Y_{U(n_2)})$. Therefore,
the maximum likelihood estimations (MLE) of ‎$\theta_1$‎ and ‎$\theta_2$ are
 ‎‎‎‎\begin{eqnarray*}
 \hat{‎\theta‎}_1=‎\frac{1}{n_1}‎X_{U(n_1)}, \ \  \hat{‎\theta‎}_2‎=‎\frac{1}{n_2}‎Y_{U(n_2)},
  \end{eqnarray*}
 and $‎\frac{2‎n_i ‎\hat{‎\theta‎}_i‎}{‎\theta_i‎}‎$ has the chi-square distribution with $2n_i$, $i=1,2$ degrees of freedom.

 Assume that  a prior information about the scale parameters is  $\theta_1=‎\theta_2‎$. Then, the pooled estimator for $‎\theta_1‎$  (and $‎\theta_2‎$)  is
 \begin{equation}\label{thetap}
 ‎\hat{‎\theta‎}_p‎=‎\frac{n_1 ‎\hat{‎\theta‎}_1+‎n_2 ‎\hat{‎\theta‎}_2}{n_1+n_2}‎.
 \end{equation}
  To incorporate the prior information into the statistical procedure, a null hypothesis regarding the information is usually formulated and tested
\cite[see][]{bancroft-44,ba-ha-77}.
Now consider testing $H_0:‎\theta_1=‎\theta_2$ against ‎‎$H_1:‎\theta_1‎ \neq‎ ‎\theta_2$. If the null hypothesis is not rejected we use the pooled samples for estimating the scale parameter, but if the null hypothesis is rejected then we use the individual sample for estimating the parameter.

The likelihood ratio test reject $H_0$ when
$‎‎\frac{‎\hat{‎\theta‎}_1‎}{‎\hat{‎\theta‎}_2‎}‎<c_1$ or $\frac{‎\hat{‎\theta‎}_1‎}{‎\hat{‎\theta‎}_2‎}>c_2$
 where
$c_1=F_{(2n_1,2n_2),\alpha/2}$, $c_2=F_{(2n_1,2n_2),1-\alpha/2}$.
Therefore,
a preliminary test estimator for $‎\theta‎_1$ may be obtained as follows
\begin{equation}‎\label{eq.thetapt}‎
\hat{\theta}_{pt}=\left\{
\begin{array}{lc}
\hat{\theta}_p & \ \ c_1< \frac{\hat{\theta}‎_1}{\hat{\theta}‎_2}<c_2\\
\hat{\theta}_1 & \ \  \ \text{otherwise}.
\end{array} \right.
\end{equation}
 Similarly, a preliminary test estimator for $‎\theta‎_2$ can be obtained as
\begin{equation}
\hat{\theta}_{pt}^*=\left\{
\begin{array}{lc}
\hat{\theta}_p & \ \ c_1< \frac{\hat{\theta}‎_1}{\hat{\theta}‎_2}<c_2\\
\hat{\theta}_2 & \ \  \ \text{otherwise}.
\end{array} \right.
\end{equation}
In the rest of this paper, we obtain the properties of preliminary test estimator $\hat{\theta}_{pt}$ and  can  use them for $\hat{\theta}_{pt}^*$.

The preliminary test estimator always depends on  the significance level ($\alpha$) of the preliminary test. The methods to seek the optimal level of significance for the preliminary test have been investigated by
\cite{sa-hi-73},
\cite{to-wa-75} and
\cite{oh-to-78}.
\cite{hirano-77}
applied AIC
\citep{akaike-98}
to determine the optimal level of significance for the preliminary test.

\begin{lemma}
For the preliminary test estimator of $‎\theta‎_1$ in \eqref{eq.thetapt},  we have

\noindent i)
\begin{eqnarray*}
Bias_{\theta_1}(\hat\theta_{pt})=E(\hat\theta_{pt})-\theta‎_1&=‎& - \lambda \left( \theta_1 [I_{d_2}(n_1+1,n_2)-I_{d_1}(n_1+1,n_2)] \right)\\
&&+ \lambda \left( \theta_2 [I_{d_2}(n_1,n_2+1)-I_{d_1}(n_1,n_2+1)] \right),
\end{eqnarray*}

\noindent ii)
\begin{eqnarray*}
MSE(\hat{\theta}_{pt})
&=&\frac{n_1+1}{n_1}\theta_1^2+(\lambda^2-2\lambda)\theta_1^2\frac{(n_1+1)}{n_1} [I_{d_2}(n_1+2,n_2)-I_{d_1}(n_1+2,n_2)]\\
&&+\lambda^2 \theta_2^2\frac{(n_2+1)}{n_2} [I_{d_2}(n_1,n_2+2)-I_{d_1}(n_1,n_2+2)]\\
&&+2(\lambda-\lambda^2)\theta_1 \theta_2 [I_{d_2}(n_1+1,n_2+1)-I_{d_1}(n_1+1,n_2+1)]-2\theta_1^2\\
&&-2 \lambda \theta_1 \theta_2 [I_{d_2}(n_1,n_2+1)-I_{d_1}(n_1,n_2+1)]\\
&&+2 \lambda \theta_1^2 [I_{d_2}(n_1+1,n_2)-I_{d_1}(n_1+1,n_2)]+\theta_1^2.
\end{eqnarray*}
where $\lambda={n_2}/{(n_1+n_2)}$,  $I_{d}(a,b)=\int_0^{d}  x^{a-1}(1-x)^{b-1}dx$ is incomplete beta function,  $d_1=1-\frac{n_2}{c_1n_1\delta}$ and $d_2=1-\frac{n_2}{c_2n_1\delta}$, and $\delta=\frac{\theta_2}{\theta_1}$.
\end{lemma}

\begin{proof}
Let $A=\left\{ \boldsymbol{X}, \boldsymbol{Y}| \ c_1<\frac{\hat{\theta}‎_1}{‎\hat{\theta}‎_2}<c_2 \right\}$ be the acceptance region.
So
\begin{eqnarray*}
\hat{\theta}_{pt}=\frac{n_1\hat{\theta}_1+n_2\hat{\theta}_2}{n_1+n_2} I_A+\hat{\theta}_1 I_{A^c}
=\hat{\theta}_1 - \lambda \hat{\theta}_1 I_A + \lambda \hat{\theta}_2 I_A,
\end{eqnarray*}
where $A^c$ is the complement of $A$. Then
\begin{eqnarray*}
E(\hat{\theta}_{pt})&=‎&\theta‎_1 - \lambda \left( \theta_1 [I_{d_2}(n_1+1,n_2)-I_{d_1}(n_1+1,n_2)] \right)\\
&&+ \lambda \left( \theta_2 [I_{d_2}(n_1,n_2+1)-I_{d_1}(n_1,n_2+1)] \right),
\end{eqnarray*}
and
\begin{eqnarray*}
E(\hat{\theta}_{pt}^2)&=&‎\frac{n_1+1}{n_1} ‎\theta_1^2 ‎‎+(\lambda^2-2\lambda) \big( \theta_1^2\frac{(n_1+1)}{n_1} [I_{d_2}(n_1+2,n_2)-I_{d_1}(n_1+2,n_2)] \big)\\
&&+\lambda^2 \big( \theta_2^2\frac{(n_2+1)}{n_2} [I_{d_2}(n_1,n_2+2)-I_{d_1}(n_1,n_2+2)] \big)\\
&&+2(\lambda-\lambda^2) \big( \theta_1 \theta_2 [I_{d_2}(n_1+1,n_2+1)-I_{d_1}(n_1+1,n_2+1)] \big).
\end{eqnarray*}
Since
$$
MSE(\hat{\theta}_{pt})=E(\hat{\theta}_{pt}^2)-2\theta_1 E(\hat{\theta}_{pt})+\theta_1^2,
$$
the proof is completed.
\end{proof}

\begin{remark}
Under the  weighted square error loss function $L(d;\theta)‎=‎\frac{(d-\theta)^2‎}{\theta‎^2}$, the  risk function is
\begin{eqnarray*}
R_{\alpha‎‎}(‎\delta‎)&=&Risk(‎\hat{\theta}‎_{pt},\theta_1‎)\\
&=&\delta^2 \big[\lambda^2 \frac{(n_2+1)}{n_2} [I_{d_2}(n_1,n_2+2)-I_{d_1}(n_1,n_2+2)]\big]\\
&&+\delta \big[2(\lambda-\lambda^2) [I_{d_2}(n_1+1,n_2+1)-I_{d_1}(n_1+1,n_2+1)]\\
&&-2\lambda [I_{d_2}(n_1,n_2+1)-I_{d_1}(n_1,n_2+1)]\big]\\
&&+\big[\frac{1}{n_1}+(\lambda^2-2\lambda) \frac{(n_1+1)}{n_1} [I_{d_2}(n_1+2,n_2)-I_{d_1}(n_1+2,n_2)]\\
&&+2 \lambda [I_{d_2}(n_1+1,n_2)-I_{d_1}(n_1+1,n_2)]\big],
\end{eqnarray*}
where $\delta=\frac{\theta_2}{\theta_1}$
\end{remark}

Notice that the risk function depends on $‎\alpha$‎ through $c_1$ and $c_2$. If $‎\delta‎\rightarrow 0$ or $\infty$‎‎‎ then $R_{\alpha‎‎}(‎\delta‎)$  converges to $R_{1}(‎\delta‎)$ which is the risk for ‎$\hat{‎\theta‎}_1$. The general shapes of $R_{\alpha‎‎}(‎\delta‎)$ for fixed values of $n_1$, $n_2$ and some $‎\alpha‎$ are shown in Figure \ref{fig.risk1}.
 Consider $\delta_1$ ‎and $\delta_2$ are intersections of $R_{0‎‎}(‎\delta‎)=({n_1+n_2\delta^2 + n_2^2 \delta^2 +n_2^2-2n_2^2 \delta})/{(n_1+n_2)^2}‎$ and $R_{1‎‎}(‎\delta‎)=‎{1}/{n_1}‎$. Therefore,  an optimal value of $\alpha$‎ is $\alpha=1$ if $\delta\leq\delta_1$‎‎‎ or $\delta\geq\delta_2$; it is $\alpha=0$, ‎otherwise. Since ‎$‎\delta‎$ is unknown, we try
 to find an optimum values  $\alpha=\alpha^{*}$ which gives a reasonable risk for all values of $\delta$‎. ‎‎‎  The regret function is defined as
‎$${Reg}(\delta,\alpha)=R_{\alpha‎‎}(‎\delta‎)-\inf_{\alpha}‎ R_{\alpha‎‎}(‎\delta‎),$$‎‎‎
where $\mathop{\inf}_{\alpha} R_{\alpha‎‎}(‎\delta‎)=R_{0‎‎}(‎\delta‎)$ for $\delta_1<\delta<\delta_2 ‎$, and $\mathop{\inf}_{\alpha} R_{\alpha‎‎}(‎\delta‎)=R_{1‎‎}(‎\delta‎)$,
otherwise.

\begin{figure}[ht]
\centering
\includegraphics[width=10cm,height=7cm]{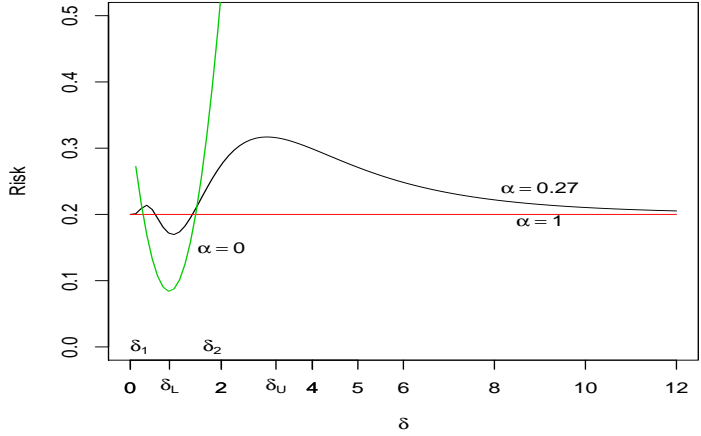}
\vspace{-0.9cm}
\caption{The risk function for some $\alpha$‎ with $n_1=5$ and $n_2=6$.}\label{fig.risk1}
\end{figure}‎

Minimizing the maximum risk sometimes leads to unreasonable or trivial results, especially in prediction problems
\cite[see][]{ho-le-50}.
 In such a situation, we are often able to arrive at reasonable results by minimizing the maximum regret instead of the maximum risk. The minimax regret criterion determines $‎\alpha‎^{*}‎$ such that
$$
\mathop{\sup}_{\delta}\{Reg(\delta,\alpha^{*}‎)\}‎‎ \leq‎ \mathop{\sup}_{\delta}\{Reg(\delta,\alpha)\},$$
for every significance level $\alpha\neq‎\alpha^{*}$.
For $‎\delta‎\leq‎\delta_2$, $Reg(\delta,\alpha)$‎‎‎ takes a maximum value at $‎\delta_L$ based on Figure \ref{fig.risk1}. Also, for $\delta>‎\delta_2$, $Reg(\delta,\alpha)$‎‎‎‎‎‎ ‎‎takes a maximum value at $‎\delta_U$.‎ Thus the minimax regret criterion determines $‎\alpha‎^{*}‎$‎ such that $Reg(\delta_L,\alpha‎^{*}‎)=Reg(\delta_U,\alpha‎^{*}‎)$. 
We found numerically the optimum significance level $‎\alpha‎^{*}‎$ for some $n_1$ and $n_2$. The results are given in Table \ref{tab.1}.

\begin{table}[htp]
\begin{center}
\caption{{\color{red} Optimal value $\alpha^*$ for some $n_1$ and $n_2$.}} \label{tab.1}
‎
\begin{tabular}{|c|cccccc|} \hline
 & \multicolumn{6}{|c|}{$n_1$} \\ \hline
 $n_{{\rm 2}}$ & 2 &  3 &  4 &  5 & 7 &  10 \\ \hline
 2 &  0.38  &  0.30 & 0.27  & 0.24 & 0.22 & 0.20 \\
 3 &  0.42  &  0.34 & 0.29  & 0.27 & 0.24 & 0.21 \\
 4 &  0.44  &  0.36 & 0.31  & 0.28 & 0.25 & 0.22 \\
 5 &  0.46  &  0.38 & 0.33  & 0.30 & 0.26 & 0.23 \\
 7 &  0.49  &  0.40 & 0.35  & 0.32 & 0.28 & 0.25 \\
10 &  0.51  &  0.42 & 0.37  & 0.33 & 0.29 & 0.26 \\ \hline
\end{tabular}

\end{center}
\end{table}

\section{Shrinkage estimator}
\label{sec.shr}

The shrinkage estimators have been discussed by a number of others, for details see
\cite{le-ca-98},
\cite{pr-si-06,pr-si-08shrinkage,pr-si-09}.
The preliminary test estimator given in Section \ref{sec.prt} uses pooled estimator $\hat{\theta}_p‎$ when preliminary test accepts the null hypothesis. Instead of using $‎\hat{‎\theta‎}_p‎$, we can use a linear combination of $\hat{\theta}_p‎$ and $\hat{\theta}_1‎$ when the preliminary test accepts $H_0$, this gives a preliminary test shrinkage estimator which is smoother than the usual preliminary test estimator. The shrinkage estimator performs better than the usual estimator when the our guess is approximately   true. Shrinkage estimators are usually defined as estimators obtained through modification of the usual (maximum likelihood, minimum variance unbiased, least squares, etc.) estimator in order to optimize some desirable criterion function.

Among various kinds of shrinkage estimators proposed so far,
\cite{jani-91}
and
\cite{kourouklis-94}
have suggested shrinkage estimators for the scale parameter in one and two-parameter exponential distributions.
In this section, we study the preliminary test shrinkage estimator $\hat{\theta}_{S}$ following the same estimation procedure by
\cite{Inada-84}:
\begin{eqnarray}‎\label{eq1}‎
\hat{\theta}_{S}= \left\{
\begin{array}{lc}
K\hat{\theta}_p+(1-K)\hat{\theta}_1 \ \ \ &  c_1<\frac{\hat{\theta}_1}{\hat{\theta}_1}<c_2\\
\hat{\theta}_1 &  \text{otherwise}.
\end{array} \right.
\end{eqnarray}

If $K=1$, $‎‎\hat{‎\theta‎‎}_{S}$ reduce to $\hat{\theta}_{pt}$. The shrinkage coefficient $K$ is not defined explicitly as a function of the test statistic. The weighting function approach is intuitively appearing, but the mean square error of the resulting estimator usually cannot be derive unless the weighting function is in some simple form. Note that $‎\hat{‎‎\theta‎‎}‎_{pt}$ approaches $‎\theta_0$ as $‎‎\alpha‎\rightarrow 0$ ‎and it approaches $‎‎\hat{‎\theta‎}‎‎‎_{Ml}$ as $‎‎\alpha‎\rightarrow 1$, however $‎\hat{‎‎\theta‎‎}‎_S$ approaches $‎\hat{‎\theta‎}‎_{pt}$ as $‎‎K\rightarrow 1$ and it approaches  $‎‎\hat{‎\theta‎}‎‎‎_{Ml}$ as $‎‎K‎\rightarrow 0$.‎
Unfortunately, different value of significance level $(‎\alpha)$ or‎ different value of shrinkage coefficient ($K$) induces a different estimator. the choice of these values depends on the decision criterion.


‎

\begin{lemma}

For the shrinkage estimator of $‎\theta‎_1$ in \eqref{eq1}‎,  we have

\noindent i)
\begin{eqnarray*}
Bias_{\theta_1}(\hat\theta_{S})=E(\hat\theta_{S})-\theta‎_1&=‎& - K \lambda \big( \theta_1 [I_{d_2}(n_1+1,n_2)-I_{d_1}(n_1+1,n_2)] \big)\\
&&+ K \lambda \big( \theta_2 [I_{d_2}(n_1,n_2+1)-I_{d_1}(n_1,n_2+1)] \big),,
\end{eqnarray*}

\noindent ii)
\begin{eqnarray*}
MSE(\hat{\theta}_S)
&=&\frac{n_1+1}{n_1}\theta_1^2+(K^2 \lambda^2-2 K \lambda)\theta_1^2\frac{(n_1+1)}{n_1} [I_{d_2}(n_1+2,n_2)-I_{d_1}(n_1+2,n_2)]\\
&&+K^2 \lambda^2 \theta_2^2\frac{(n_2+1)}{n_2} [I_{d_2}(n_1,n_2+2)-I_{d_1}(n_1,n_2+2)]\\
&&+2(K \lambda-K^2 \lambda^2)\theta_1 \theta_2 [I_{d_2}(n_1+1,n_2+1)-I_{d_1}(n_1+1,n_2+1)]-2\theta_1^2\\
&&-2 K \lambda \theta_1 \theta_2 [I_{d_2}(n_1,n_2+1)-I_{d_1}(n_1,n_2+1)]\\
&&+2 K \lambda \theta_1^2 [I_{d_2}(n_1+1,n_2)-I_{d_1}(n_1+1,n_2)]+\theta_1^2.
\end{eqnarray*}
\end{lemma}

\begin{proof}
The straightforward evaluation leads to
\begin{eqnarray*}
E(\hat{\theta}_S)&=&‎\theta‎_1 - K \lambda \big( \theta_1 [I_{d_2}(n_1+1,n_2)-I_{d_1}(n_1+1,n_2)] \big)\\
&&+ K \lambda \big( \theta_2 [I_{d_2}(n_1,n_2+1)-I_{d_1}(n_1,n_2+1)] \big),
\end{eqnarray*}
and
 \begin{eqnarray*}
E(\hat{\theta}_{S}^2)&=&‎\frac{n_1+1}{n_1} ‎\theta_1^2 ‎‎+(K^2 \lambda^2-2 K \lambda) \big( \theta_1^2\frac{(n_1+1)}{n_1} [I_{d_2}(n_1+2,n_2)-I_{d_1}(n_1+2,n_2)] \big)\\
&&+K^2 \lambda^2 \big( \theta_2^2\frac{(n_2+1)}{n_2} [I_{d_2}(n_1,n_2+2)-I_{d_1}(n_1,n_2+2)] \big)\\
&&+2(K \lambda- K^2 \lambda^2) \big( \theta_1 \theta_2 [I_{d_2}(n_1+1,n_2+1)-I_{d_1}(n_1+1,n_2+1)] \big).
\end{eqnarray*}
 Since
$$
MSE(\hat{\theta}_{S})=E(\hat{\theta}_{S}^2)-2\theta_1 E(\hat{\theta}_{S})+\theta_1^2,
$$
the proof is completed.
\end{proof}

\begin{remark}
Under the weighted square error loss function $L(d;‎\theta)‎=‎\frac{(d-‎\theta)^2‎}{‎\theta‎^2}$, we have
\begin{eqnarray*}
R_{\alpha}(‎\delta,K)&=&‎‎Risk(‎\hat{‎\theta‎}‎_S,\theta_1)\\
&=&\delta^2 \big[K^2 \lambda^2 \frac{(n_2+1)}{n_2} [I_{d_2}(n_1,n_2+2)-I_{d_1}(n_1,n_2+2)]\big]\\
&&+\delta \big[2(K \lambda-K^2 \lambda^2) [I_{d_2}(n_1+1,n_2+1)-I_{d_1}(n_1+1,n_2+1)]\\
&&-2 K\lambda [I_{d_2}(n_1,n_2+1)-I_{d_1}(n_1,n_2+1)]\big]\\
&&+\big[\frac{1}{n_1}+(K^2 \lambda^2-2 K \lambda) \frac{(n_1+1)}{n_1} [I_{d_2}(n_1+2,n_2)-I_{d_1}(n_1+2,n_2)]\\
&&+2 K \lambda [I_{d_2}(n_1+1,n_2)-I_{d_1}(n_1+1,n_2)]\big].
\end{eqnarray*}
\end{remark}

The regret function is defined as
$$Reg(‎\delta,K,‎\alpha)=R_{\alpha}(‎\delta,K)-\inf_K R_{\alpha}(‎\delta,K).$$
Since $R_{\alpha}(‎\delta,K)$ has quadratic form w.r.t $K$, so it has minimum in the interval [0,1]. Let $R_{\alpha}(‎\delta,K)=K^2 h_2(‎\delta‎)‎‎+ K h_1(‎\delta‎) + h_0(‎\delta‎)$ then ‎$\frac{‎\partial R‎}{‎‎\partial K‎‎}=0‎$ imply that $K_0=‎\frac{-h_1(‎\delta‎)}{2 h_2(‎\delta‎)}‎$‎‎ is an extremum point. Therefore
$$ \inf_K R_{\alpha}(‎\delta,K) =\left\{
\begin{array}{lc}
  \min\{R_{\alpha}(‎\delta,0), R_{\alpha}(‎\delta,1), R_{\alpha}(‎\delta,K_0)\}‎‎‎ \ \   & \text{if} \ \ K_0‎‎ \in (0,1)‎‎\\
 \min\{ R_{\alpha}(‎\delta,0), R_{\alpha}(‎\delta,1)\}‎‎‎  & \text{otherwise}.
\end{array}\right.$$‎
After rather extensive numerical investigation the values $K$ which attain the $\inf_K \sup_{‎\delta‎}Reg(‎\delta, $ $K, ‎\alpha‎‎)$ are obtained.
We plot the risk function for $n_1=5$, $n_2=6$, $‎\alpha=0.16‎$, for $K=0, 1, 0.21$ in Figure \ref{risk3}.
Let $‎\delta_1$ and $‎\delta_2$ be intersections of $R_{\alpha}(‎\delta,0)$ and $R_{\alpha}(‎\delta,1)$. For $‎\delta‎\leq‎\delta_2$‎‎‎, $Reg(‎\delta,K,‎\alpha‎)$‎ takes a maximum value at $‎\delta_L$‎.  For $‎\delta‎>‎\delta_2$‎‎‎,  $Reg(‎\delta,K,‎\alpha‎)$‎  takes a maximum value at $‎\delta_U$.  The reasonable estimator  is obtained when these two maximum values are equal. Thus the minimax regret criterion determines $K^*$ such that: $$Reg(‎\delta_L,K^*‎) =Reg(‎\delta_U,K^*‎).$$

To find the optimal value of $K$, two cases are considered for $‎\alpha$‎:

\noindent {\bf Case I}: Let $‎\alpha=0.16‎$, that is the AIC optimal level of significance
\citep[see][]{Inada-84},
which is independent of $n$. Table \ref{tab.2} presents the values of $K^*$ for some $n_1$ and $n_2$.‎

\noindent {\bf Case II}: ‎Let $‎\alpha=‎\alpha^*$ ‎equal the significance level of pre-test. Table \ref{tab.3} presents the values of $K^*$ in this case for some $n_1$ and $n_2$.

\begin{figure}
 \begin{center}
\includegraphics[width=10cm,height=7cm]{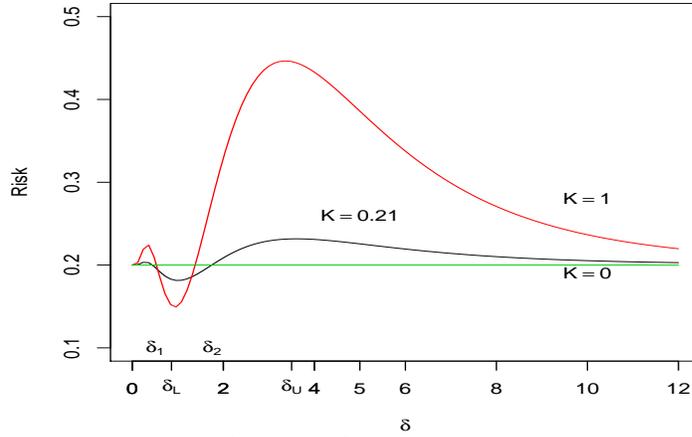}
\vspace{-1cm}
\caption{  The risk function for some $K‎$ with $n_1=5$ and $n_2=6$.}\label{risk3}
\end{center}
\end{figure}
‎

\begin{table}[htp]
\begin{center}
\caption{Optimal value $K$ for some $n_1$ and $n_2$ with $‎\alpha=0.16‎$.‎‎}\label{tab.2}
‎
\begin{tabular}{|c|cccccc|} \hline
 & \multicolumn{6}{|c|}{$n_1$} \\ \hline
  $n_2$ &  2 &  3 & 4 &  5 &  7 &  10 \\ \hline
2 & 0.17 & 0.23 & 0.29 & 0.32 & 0.38 & 0.42 \\
3 & 0.14 & 0.19 & 0.24 & 0.27 & 0.32 & 0.36 \\
4 & 0.12 & 0.17 & 0.21 & 0.24 & 0.29 & 0.33 \\
5 & 0.11 & 0.16 & 0.20 & 0.22 & 0.27 & 0.31 \\
7 & 0.10 & 0.14 & 0.18 & 0.20 & 0.24 & 0.28 \\
10 & 0.10 & 0.13 & 0.17 & 0.19 & 0.23 & 0.26 \\ \hline

\end{tabular}

\end{center}

\end{table}

\begin{table}[htp]
\begin{center}
\caption{Optimal value $K$ for some $n_1$ and $n_2$ with optimal $‎\alpha=‎\alpha^*‎‎$‎‎‎‎.}\label{tab.3}

\begin{tabular}{|c|cc|cc|cc|cc|cc|cc|} \hline
 & \multicolumn{2}{|c|}{$n_{1}=2$} & \multicolumn{2}{|c|}{$n_1=3$} & \multicolumn{2}{|c|}{$n_1=4$} & \multicolumn{2}{|c|}{$n_1=5$} & \multicolumn{2}{|c|}{$n_1=7$} & \multicolumn{2}{|c|}{$n_1=10$} \\ \hline
$n_2$ & ${\alpha }^*$ & $K^*$ & ${\alpha }^*$ & $K^*$ & ${\alpha }^*$ & $K^*$ & ${\alpha }^*$ & $K^*$ & ${\alpha }^*$ & $K^*$ & ${\alpha }^*$ & $K^*$ \\ \hline
2 & 0.38 & 0.21 & 0.30 & 0.29 & 0.27 & 0.34 & 0.24 & 0.37 & 0.22 & 0.42 & 0.20 & 0.45 \\
3 & 0.42 & 0.15 & 0.34 & 0.23 & 0.29 & 0.28 & 0.27 & 0.31 & 0.24 & 0.35 & 0.21 & 0.39 \\
4 & 0.44 & 0.12 & 0.36 & 0.19 & 0.31 & 0.24 & 0.28 & 0.27 & 0.25 & 0.32 & 0.22 & 0.35 \\
5 & 0.46 & 0.11 & 0.38 & 0.17 & 0.33 & 0.22 & 0.30 & 0.25 & 0.26 & 0.29 & 0.23 & 0.33 \\
7 & 0.49 & 0.10 & 0.40 & 0.15 & 0.35 & 0.19 & 0.32 & 0.22 & 0.28 & 0.26 & 0.25 & 0.30 \\
10 & 0.51 & 0.09 & 0.42 & 0.14 & 0.37 & 0.17 & 0.33 & 0.20 & 0.29 & 0.24 & 0.26 & 0.28 \\ \hline
\end{tabular}

\end{center}
\end{table}

\begin{remark}
For the general case, if we observe upper record values $X_{U(1)}, \dots, X_{U(n_1)}$ and
 $Y_{U(1)}, \dots, Y_{U(n_2)}$ from location-scale exponential distribution
with the following probability density function
\begin{eqnarray*}
 f_X(x)=\frac{1}{\theta_1} \exp ( - \frac{x-\eta_1}{\theta_1} ) \ \ \ \ x > \eta_1,    \ \ \ \
 f_Y(y)=\frac{1}{\theta_2} \exp ( - \frac{y-\eta_2}{\theta_2} ) \ \ \ \ y > \eta_2,
 \end{eqnarray*}
the MLE's of ‎$\theta_1$‎ and $\theta_2$ are
‎\begin{eqnarray*}
\hat{‎\theta‎}_1=‎\frac{1}{n_1}(‎X_{U(n_1)}-X_{U(1)}), \ \  \hat{‎\theta‎}_2‎=‎\frac{1}{n_2}(‎Y_{U(n_2)}-‎Y_{U(1)}),
\end{eqnarray*}
 and ${2‎n_i ‎\hat{‎\theta‎}_i‎}/{‎\theta_i‎}‎$ has the chi-square distribution with $2n_i-2$, $i=1,2$ degrees of freedom.
In this case, we have the same result as the special case
in absence of location parameter with a new statistic and degree of freedom.
\end{remark}

\section{Numerical study}
\label{sec.num}
In this section, we first illustrate the proposed estimators using an example. Then, these estimators are compared using simulation.

\subsection{An example}
The following the simulated record values are given by \cite{baklizi-13}. The first record values which are generated from location-scale exponential distribution with $\theta_1=1$ and $\eta_1=3$ are
3.105, 6.158, 6.296, 6.824, 7.282, 10.200, 10.240, and 11.669,
and
first record values which are generated from location-scale exponential distribution with $\theta_2=1$ and $\eta_2=1$ are 1.177, 2.430, 4.090, 4.349, 4.624, 5.655, 6.021, and 6.987.

The MLE's of $\theta_1$ and $\theta_2$ are $\hat\theta_1=1.0705$ and $\hat\theta_2=0.7262$. Since the hypothesis $H_0: \theta_1=\theta_2$ is not rejected the preliminary  test  and optimal shrinkage estimations are  $\hat\theta_{pt}=0.8984$, and $\hat\theta_{S}=1.0292$, respectively. It can be seen that the optimal shrinkage estimation is very close to the true value $\theta_1=1$.

\subsection{Simulation}
We performed a simulation to compare the ML, preliminary test, and shrinkage estimators with 10000  repetition. First, a random sample with size $n_1$ is generated from an exponential distribution with scale parameter $\theta_1=1$. Then, a random sample with size $n_2$ is generated from an exponential distribution with scale parameter $\theta_2$. We consider some values between 0.1 and 3.0 for $\theta_2$. The estimators are calculated and  their biases are obtained. Also, the efficiency  of   preliminary test, and shrinkage estimators with respect to MLE are computed (Here, the efficiency is ratio of MSEs). The results for $\alpha=0.16$ are  given in Tables \ref{tab.sim1},
and for the optimal $\alpha$ and $K$ (proposed in Table \ref{tab.3})  are given in Table \ref{tab.sim2}.

It can be concluded that\\
i.  When $\theta_1$ and $\theta_2$ are very close, the preliminary test estimator is better than the shrinkage estimator. Also,  the shrinkage estimator is better than the MLE.  Therefore, the two proposed estimators are better that MLE when a prior information about the scale parameters is  $\theta_1=\theta_2$.   \\
ii. When $\theta_1$ and $\theta_2$ are not close, the MLE are better than the shrinkage estimator. Also, the shrinkage estimator is better than  the preliminary test estimator.

\begin{table}
\begin{center}
\caption{The simulation results with $\alpha=0.16$.} \label{tab.sim1}
\begin{tabular}{|c|c|ccc|cc||c|c|ccc|cc|} \hline
 &  &  \multicolumn{3}{|c|}{Bias} & \multicolumn{2}{|c||}{Efficiency} &  &    & \multicolumn{3}{|c|}{Bias} & \multicolumn{2}{|c|}{Efficiency} \\ \hline

$n_1,n_2$ & $\theta_2$ & $\hat\theta_1$ & $\hat\theta_{pt}$ & $\hat\theta_s$ & $\hat\theta_{pt}$ & $\hat\theta_s$ &
$n_1,n_2$ & $\theta_2$ & $\hat\theta_1$ & $\hat\theta_{pt}$ & $\hat\theta_s$ & $\hat\theta_{pt}$ & $\hat\theta_s$ \\ \hline

2,2 & 0.1 & -0.003 & -0.003 & -0.003 & 1.000 & 1.000 & 7,2 & 0.1 & 0.002 & 0.000 & 0.000 & 0.999 & 1.000 \\
 & 0.3 & -0.003 & -0.022 & -0.008 & 0.902 & 0.978 &  & 0.3 & 0.002 & -0.018 & -0.005 & 0.905 & 0.978 \\
 & 0.5 & -0.003 & -0.056 & -0.017 & 0.914 & 0.995 &  & 0.5 & 0.002 & -0.054 & -0.014 & 0.913 & 0.994 \\
 & 0.8 & -0.003 & -0.042 & -0.013 & 1.200 & 1.080 &  & 0.8 & 0.002 & -0.036 & -0.009 & 1.201 & 1.079 \\
 & 1.0 & -0.003 & -0.006 & -0.004 & 1.240 & 1.095 &  & 1.0 & 0.002 & 0.000 & 0.000 & 1.226 & 1.091 \\
 & 1.2 & -0.003 & 0.028 & 0.007 & 1.099 & 1.069 &  & 1.2 & 0.002 & 0.032 & 0.010 & 1.104 & 1.070 \\
 & 1.5 & -0.003 & 0.064 & 0.014 & 0.849 & 1.005 &  & 1.5 & 0.002 & 0.067 & 0.017 & 0.878 & 1.016 \\
 & 2.0 & -0.003 & 0.081 & 0.019 & 0.661 & 0.930 &  & 2.0 & 0.002 & 0.080 & 0.021 & 0.682 & 0.938 \\
 & 2.5 & -0.003 & 0.068 & 0.016 & 0.635 & 0.911 &  & 2.5 & 0.002 & 0.071 & 0.018 & 0.643 & 0.913 \\
 & 3.0 & -0.003 & 0.050 & 0.011 & 0.667 & 0.917 &  & 3.0 & 0.002 & 0.054 & 0.014 & 0.662 & 0.915 \\ \hline

2,5 & 0.1 & -0.001 & 0.002 & 0.002 & 1.000 & 1.000 & 7,5 & 0.1 & 0.001 & 0.003 & 0.003 & 0.999 & 1.000 \\
 & 0.3 & -0.001 & -0.016 & -0.002 & 0.903 & 0.977 &  & 0.3 & 0.001 & -0.017 & -0.002 & 0.901 & 0.977 \\
 & 0.5 & -0.001 & -0.050 & -0.011 & 0.925 & 0.997 &  & 0.5 & 0.001 & -0.051 & -0.011 & 0.922 & 0.997 \\
 & 0.8 & -0.001 & -0.036 & -0.008 & 1.201 & 1.078 &  & 0.8 & 0.001 & -0.036 & -0.007 & 1.207 & 1.080 \\
 & 1.0 & -0.001 & 0.000 & 0.002 & 1.244 & 1.094 &  & 1.0 & 0.001 & 0.001 & 0.002 & 1.230 & 1.091 \\
 & 1.2 & -0.001 & 0.032 & 0.010 & 1.112 & 1.075 &  & 1.2 & 0.001 & 0.035 & 0.012 & 1.093 & 1.066 \\
 & 1.5 & -0.001 & 0.069 & 0.020 & 0.857 & 1.007 &  & 1.5 & 0.001 & 0.068 & 0.020 & 0.873 & 1.013 \\
 & 2.0 & -0.001 & 0.084 & 0.024 & 0.684 & 0.938 &  & 2.0 & 0.001 & 0.087 & 0.025 & 0.671 & 0.934 \\
 & 2.5 & -0.001 & 0.076 & 0.021 & 0.636 & 0.910 &  & 2.5 & 0.001 & 0.073 & 0.021 & 0.646 & 0.915 \\
 & 3.0 & -0.001 & 0.055 & 0.016 & 0.673 & 0.919 &  & 3.0 & 0.001 & 0.055 & 0.017 & 0.664 & 0.915 \\ \hline

2,10 & 0.1 & -0.001 & -0.001 & -0.001 & 0.999 & 1.000 & 7,7 & 0.1 & -0.005 & 0.001 & 0.001 & 1.000 & 1.000 \\
 & 0.3 & -0.001 & -0.019 & -0.005 & 0.903 & 0.978 &  & 0.3 & -0.005 & -0.018 & -0.004 & 0.908 & 0.979 \\
 & 0.5 & -0.001 & -0.055 & -0.015 & 0.922 & 0.997 &  & 0.5 & -0.005 & -0.052 & -0.013 & 0.914 & 0.994 \\
 & 0.8 & -0.001 & -0.038 & -0.010 & 1.191 & 1.077 &  & 0.8 & -0.005 & -0.038 & -0.009 & 1.198 & 1.079 \\
 & 1.0 & -0.001 & -0.003 & -0.001 & 1.216 & 1.088 &  & 1.0 & -0.005 & -0.002 & 0.000 & 1.219 & 1.089 \\
 & 1.2 & -0.001 & 0.031 & 0.007 & 1.075 & 1.063 &  & 1.2 & -0.005 & 0.030 & 0.006 & 1.094 & 1.068 \\
 & 1.5 & -0.001 & 0.064 & 0.016 & 0.880 & 1.014 &  & 1.5 & -0.005 & 0.067 & 0.018 & 0.861 & 1.008 \\
 & 2.0 & -0.001 & 0.080 & 0.020 & 0.690 & 0.940 &  & 2.0 & -0.005 & 0.081 & 0.022 & 0.683 & 0.938 \\
 & 2.5 & -0.001 & 0.072 & 0.018 & 0.638 & 0.912 &  & 2.5 & -0.005 & 0.071 & 0.019 & 0.645 & 0.914 \\
 & 3.0 & -0.001 & 0.055 & 0.014 & 0.654 & 0.912 &  & 3.0 & -0.005 & 0.056 & 0.015 & 0.656 & 0.913 \\ \hline

10,2 & 0.1 & -0.004 & -0.002 & -0.002 & 1.000 & 1.000 & 10,7 & 0.1 & 0.002 & 0.004 & 0.004 & 0.999 & 1.000 \\
 & 0.3 & -0.004 & -0.020 & -0.007 & 0.903 & 0.978 &  & 0.3 & 0.002 & -0.015 & -0.001 & 0.903 & 0.978 \\
 & 0.5 & -0.004 & -0.054 & -0.015 & 0.916 & 0.995 &  & 0.5 & 0.002 & -0.049 & -0.010 & 0.920 & 0.996 \\
 & 0.8 & -0.004 & -0.036 & -0.011 & 1.206 & 1.082 &  & 0.8 & 0.002 & -0.035 & -0.006 & 1.207 & 1.081 \\
 & 1.0 & -0.004 & -0.004 & -0.003 & 1.221 & 1.090 &  & 1.0 & 0.002 & 0.000 & 0.003 & 1.246 & 1.094 \\
 & 1.2 & -0.004 & 0.031 & 0.008 & 1.100 & 1.071 &  & 1.2 & 0.002 & 0.034 & 0.010 & 1.092 & 1.067 \\
 & 1.5 & -0.004 & 0.064 & 0.015 & 0.861 & 1.011 &  & 1.5 & 0.002 & 0.069 & 0.021 & 0.862 & 1.009 \\
 & 2.0 & -0.004 & 0.080 & 0.020 & 0.674 & 0.935 &  & 2.0 & 0.002 & 0.087 & 0.025 & 0.680 & 0.938 \\
 & 2.5 & -0.004 & 0.069 & 0.017 & 0.632 & 0.909 &  & 2.5 & 0.002 & 0.077 & 0.023 & 0.629 & 0.907 \\
 & 3.0 & -0.004 & 0.053 & 0.013 & 0.653 & 0.911 &  & 3.0 & 0.002 & 0.060 & 0.018 & 0.664 & 0.916 \\ \hline

\end{tabular}
\end{center}
\end{table}

\begin{table}
\begin{center}
\caption{The simulation results with optimal $\alpha$ and $K$.} \label{tab.sim2}
\begin{tabular}{|c|c|ccc|cc||c|c|ccc|cc|} \hline
 &  &  \multicolumn{3}{|c|}{Bias} & \multicolumn{2}{|c||}{Efficiency} &  &    & \multicolumn{3}{|c|}{Bias} & \multicolumn{2}{|c|}{Efficiency} \\ \hline

$n_1,n_2$ & $\theta_2$ & $\hat\theta_1$ & $\hat\theta_{pt}$ & $\hat\theta_s$ & $\hat\theta_{pt}$ & $\hat\theta_s$ &
$n_1,n_2$ & $\theta_2$ & $\hat\theta_1$ & $\hat\theta_{pt}$ & $\hat\theta_s$ & $\hat\theta_{pt}$ & $\hat\theta_s$ \\ \hline

2,2 & 0.1 & -0.003 & -0.003 & -0.003 & 1.000 & 1.000 & 7,2 & 0.1 & 0.002 & 0.002 & 0.002 & 1.000 & 1.000 \\
 & 0.3 & -0.003 & -0.007 & -0.004 & 0.975 & 0.993 &  & 0.3 & 0.002 & -0.009 & -0.001 & 0.936 & 0.983 \\
 & 0.5 & -0.003 & -0.019 & -0.008 & 0.965 & 0.994 &  & 0.5 & 0.002 & -0.035 & -0.009 & 0.938 & 0.993 \\
 & 0.8 & -0.003 & -0.017 & -0.007 & 1.058 & 1.028 &  & 0.8 & 0.002 & -0.026 & -0.006 & 1.128 & 1.062 \\
 & 1.0 & -0.003 & -0.006 & -0.004 & 1.066 & 1.033 &  & 1.0 & 0.002 & 0.000 & 0.002 & 1.160 & 1.076 \\
 & 1.2 & -0.003 & 0.013 & 0.007 & 1.027 & 1.023 &  & 1.2 & 0.002 & 0.020 & 0.004 & 1.063 & 1.054 \\
 & 1.5 & -0.003 & 0.015 & 0.002 & 0.961 & 1.002 &  & 1.5 & 0.002 & 0.049 & 0.016 & 0.904 & 1.006 \\
 & 2.0 & -0.003 & 0.017 & 0.003 & 0.888 & 0.975 &  & 2.0 & 0.002 & 0.058 & 0.019 & 0.754 & 0.943 \\
 & 2.5 & -0.003 & 0.013 & 0.002 & 0.885 & 0.970 &  & 2.5 & 0.002 & 0.049 & 0.016 & 0.725 & 0.924 \\
 & 3.0 & -0.003 & 0.007 & 0.000 & 0.910 & 0.977 &  & 3.0 & 0.002 & 0.038 & 0.013 & 0.750 & 0.930 \\ \hline

2,5 & 0.1 & -0.001 & -0.001 & -0.001 & 1.000 & 1.000 & 7,5 & 0.1 & 0.001 & 0.001 & 0.001 & 1.000 & 1.000 \\
 & 0.3 & -0.001 & -0.002 & -0.001 & 0.987 & 0.996 &  & 0.3 & 0.001 & -0.007 & -0.002 & 0.951 & 0.987 \\
 & 0.5 & -0.001 & -0.010 & -0.004 & 0.975 & 0.995 &  & 0.5 & 0.001 & -0.029 & -0.008 & 0.939 & 0.991 \\
 & 0.8 & -0.001 & -0.011 & -0.004 & 1.036 & 1.018 &  & 0.8 & 0.001 & -0.023 & -0.006 & 1.107 & 1.052 \\
 & 1.0 & -0.001 & -0.004 & -0.002 & 1.051 & 1.024 &  & 1.0 & 0.001 & -0.002 & 0.000 & 1.148 & 1.068 \\
 & 1.2 & -0.001 & 0.003 & 0.001 & 1.031 & 1.018 &  & 1.2 & 0.001 & 0.022 & 0.010 & 1.060 & 1.046 \\
 & 1.5 & -0.001 & 0.011 & 0.003 & 0.979 & 1.003 &  & 1.5 & 0.001 & 0.038 & 0.012 & 0.918 & 1.003 \\
 & 2.0 & -0.001 & 0.012 & 0.003 & 0.930 & 0.984 &  & 2.0 & 0.001 & 0.046 & 0.014 & 0.786 & 0.950 \\
 & 2.5 & -0.001 & 0.009 & 0.002 & 0.929 & 0.982 &  & 2.5 & 0.001 & 0.037 & 0.012 & 0.780 & 0.942 \\
 & 3.0 & -0.001 & 0.006 & 0.001 & 0.944 & 0.985 &  & 3.0 & 0.001 & 0.027 & 0.009 & 0.799 & 0.945 \\ \hline

2,10 & 0.1 & -0.001 & -0.001 & -0.001 & 1.000 & 1.000 & 7,7 & 0.1 & -0.005 & -0.005 & -0.005 & 1.000 & 1.000 \\
 & 0.3 & -0.001 & -0.002 & -0.001 & 0.990 & 0.997 &  & 0.3 & -0.005 & -0.011 & -0.007 & 0.956 & 0.988 \\
 & 0.5 & -0.001 & -0.008 & -0.003 & 0.980 & 0.996 &  & 0.5 & -0.005 & -0.031 & -0.013 & 0.941 & 0.990 \\
 & 0.8 & -0.001 & -0.008 & -0.003 & 1.020 & 1.011 &  & 0.8 & -0.005 & -0.027 & -0.011 & 1.104 & 1.050 \\
 & 1.0 & -0.001 & -0.003 & -0.001 & 1.028 & 1.015 &  & 1.0 & -0.005 & -0.007 & -0.005 & 1.126 & 1.060 \\
 & 1.2 & -0.003 & -0.003 & -0.002 & 1.019 & 1.012 &  & 1.2 & -0.005 & 0.013 & 0.001 & 1.052 & 1.042 \\
 & 1.5 & -0.001 & 0.007 & 0.002 & 0.982 & 1.001 &  & 1.5 & -0.005 & 0.031 & 0.006 & 0.912 & 0.999 \\
 & 2.0 & -0.001 & 0.008 & 0.002 & 0.951 & 0.989 &  & 2.0 & -0.005 & 0.034 & 0.007 & 0.806 & 0.955 \\
 & 2.5 & -0.001 & 0.005 & 0.001 & 0.953 & 0.988 &  & 2.5 & -0.005 & 0.027 & 0.005 & 0.793 & 0.945 \\
 & 3.0 & -0.001 & 0.004 & 0.001 & 0.954 & 0.988 &  & 3.0 & -0.005 & 0.018 & 0.002 & 0.826 & 0.953 \\ \hline

10,2 & 0.1 & -0.004 & -0.004 & -0.004 & 1.000 & 1.000 & 10,7 & 0.1 & 0.002 & 0.002 & 0.002 & 1.000 & 1.000 \\
 & 0.3 & -0.004 & -0.018 & -0.008 & 0.926 & 0.980 &  & 0.3 & 0.002 & -0.006 & -0.001 & 0.950 & 0.986 \\
 & 0.5 & -0.004 & -0.046 & -0.017 & 0.923 & 0.990 &  & 0.5 & 0.002 & -0.029 & -0.008 & 0.939 & 0.991 \\
 & 0.8 & -0.004 & -0.037 & -0.014 & 1.156 & 1.072 &  & 0.8 & 0.002 & -0.024 & -0.006 & 1.122 & 1.056 \\
 & 1.0 & -0.004 & -0.008 & -0.006 & 1.193 & 1.089 &  & 1.0 & 0.002 & -0.003 & 0.000 & 1.144 & 1.068 \\
 & 1.2 & -0.004 & 0.028 & 0.010 & 1.089 & 1.066 &  & 1.2 & 0.002 & 0.014 & 0.002 & 1.072 & 1.051 \\
 & 1.5 & -0.004 & 0.048 & 0.011 & 0.893 & 1.008 &  & 1.5 & 0.002 & 0.043 & 0.014 & 0.904 & 1.001 \\
 & 2.0 & -0.004 & 0.058 & 0.014 & 0.738 & 0.941 &  & 2.0 & 0.002 & 0.049 & 0.016 & 0.785 & 0.951 \\
 & 2.5 & -0.004 & 0.048 & 0.011 & 0.709 & 0.921 &  & 2.5 & 0.002 & 0.040 & 0.013 & 0.766 & 0.937 \\
 & 3.0 & -0.004 & 0.035 & 0.007 & 0.729 & 0.923 &  & 3.0 & 0.002 & 0.028 & 0.010 & 0.802 & 0.946 \\ \hline

\end{tabular}

\end{center}
\end{table}

\section*{Acknowledgements}
The authors would like to thank the anonymous  referees  for many helpful comments and  suggestions.

\newpage

\bibliographystyle{apa}

\end{document}